\documentclass[11pt]{article}
\usepackage[T1]{fontenc}
\usepackage[utf8]{inputenc}
\usepackage{lmodern}
\usepackage[a4paper,margin=28mm]{geometry}
\usepackage{amsmath,amssymb,amsthm,amscd,mathtools}
\usepackage{microtype}
\usepackage{enumitem}
\usepackage{xcolor}
\usepackage{hyperref}
\hypersetup{colorlinks=true,linkcolor=blue!45!black,citecolor=blue!45!black,
  urlcolor=blue!45!black,pdftitle={A proof of Liu--Xu Conjecture 5.2},
  pdfauthor={}}
\setlist[enumerate]{leftmargin=*,itemsep=5pt}
\setlist[itemize]{leftmargin=*,itemsep=4pt}
\numberwithin{equation}{section}
\newtheorem{theorem}{Theorem}[section]
\newtheorem{proposition}[theorem]{Proposition}
\newtheorem{lemma}[theorem]{Lemma}
\newtheorem{corollary}[theorem]{Corollary}
\theoremstyle{definition}
\newtheorem{definition}[theorem]{Definition}
\theoremstyle{remark}
\newtheorem{remark}[theorem]{Remark}
\newcommand{\Mbar}{\overline{\mathcal M}}
\newcommand{\EE}{\mathbb E}
\newcommand{\OO}{\mathcal O}
\newcommand{\CC}{\mathbb C}
\newcommand{\QQ}{\mathbb Q}
\newcommand{\RR}{\mathbb R}

\newcommand{\rk}{\operatorname{rk}}
\newcommand{\pr}{\operatorname{pr}}
\newcommand{\codim}{\operatorname{codim}}
\newcommand{\id}{\operatorname{id}}

\title{Balancing tautological intersection numbers}
\author{Kefeng Liu \\ Chongqing University of Technology\\ kefengliu@cqut.edu.cn
\and 
Hao Xu \\ Chongqing University of Technology\\ haoxu@cqut.edu.cn}
\date{}

\begin{document}
\maketitle

\begin{abstract}
We prove the three balancing inequalities conjectured by the authors about 20 years ago: moving one unit from a larger index to a smaller one cannot
decrease a mixed tautological intersection number, whether the two
distinguished insertions are psi classes, kappa classes, or Hodge Chern
classes. 
\end{abstract}

\tableofcontents

\section{The conjecture, conventions, and main theorem}
\label{sec:statement}

All spaces and stacks are over $\CC$. Fix nonnegative integers $g,n$
with $2g-2+n>0$, and write
\[
  X=\Mbar_{g,n},\qquad D=3g-3+n.
\]
Intersection numbers on $X$ mean degrees in the rational Chow theory of
the proper Deligne--Mumford stack, as developed in \cite{Vistoli};
in particular, the usual automorphism factors are included.
We consider products of total
codimension $D$. A product not of top degree contributes zero to the
pushforward to a point.

Let $\mathcal L_i$ be the $i$th cotangent line bundle and
$\psi_i=c_1(\mathcal L_i)$. Let $\EE$ be the rank-$g$ Hodge bundle,
and set
\[
  \lambda_b=c_b(\EE),\qquad
  \lambda_0=1,\qquad \lambda_b=0\quad(b>g).
\]
More generally, for any vector bundle $E$ we use $c_0(E)=1$ and
$c_j(E)=0$ for $j<0$ or $j>\rk E$.

Write $\pi:C\to X$ for the universal curve, with marked sections
having image divisors $D_1,\ldots,D_n$. We use the logarithmic kappa
classes \cite{AC}
\begin{equation}\label{eq:kappa-definition}
  L=\omega_\pi(D_1+\cdots+D_n),\qquad
  \ell=c_1(L),\qquad
  \kappa_a=\pi_*(\ell^{a+1})\quad(a\ge0).
\end{equation}
Under the identification $C=\Mbar_{g,n+1}$, the class $\ell$ is
$\psi_{n+1}$; this will be discussed in Section \ref{subsec:log-line}.
The integer $n$ always counts all original markings; in particular,
two distinguished psi insertions occupy two of these $n$ markings.

The following is conjectured by the authors in \cite[Conjecture 5.2]{LX}, proved here in Section \ref{sec:proof}.

\begin{theorem}\label{thm:main}
Let $0\le p<q$. All additional indices in the following statements
are nonnegative integers, and the indicated products have total
codimension $D$.
\begin{enumerate}[label=\textup{(\roman*)}]
\item Suppose $n\ge2$, and set
\[
  P_\psi=\prod_{i=3}^n\psi_i^{d_i}
          \prod_{j=1}^m\kappa_{a_j}
          \prod_{t=1}^k\lambda_{b_t}.
\]
Then
\begin{equation}\label{eq:main-psi}
  \int_X\psi_1^p\psi_2^qP_\psi
  \le \int_X\psi_1^{p+1}\psi_2^{q-1}P_\psi.
\end{equation}
\item Set
\[
  P=\prod_{i=1}^n\psi_i^{d_i}
       \prod_{j=1}^m\kappa_{a_j}
       \prod_{t=1}^k\lambda_{b_t}.
\]
Then
\begin{equation}\label{eq:main-kappa}
  \int_X\kappa_p\kappa_qP
  \le \int_X\kappa_{p+1}\kappa_{q-1}P.
\end{equation}
\item With $P$ as in \textup{(ii)},
\begin{equation}\label{eq:main-lambda}
  \int_X\lambda_p\lambda_qP
  \le \int_X\lambda_{p+1}\lambda_{q-1}P.
\end{equation}
\end{enumerate}
Empty products are $1$, and repeated insertions are allowed.
\end{theorem}

These are weak inequalities. In particular, when $q=p+1$ they are
equalities, by symmetry or commutativity. No general strictness
assertion is intended. Faber \cite{Faber} had an algorithm calculating these integrals
based on the Witten-Kontsevich Theorem \cite{Kontsevich, Witten}.

Recently, Schmitt \cite{Schmitt} proved the pure-psi case:
\begin{equation}
  \int_X\psi_1^p\psi_2^q \prod_{i=3}^n\psi_i^{d_i}
  \le \int_X\psi_1^{p+1}\psi_2^{q-1} \prod_{i=3}^n\psi_i^{d_i}
\end{equation}
when $0\le p<q$. This is a special case of \eqref{eq:main-psi}.

Our work was inspired by the balancing mechanism of Schmitt \cite{Schmitt}.
The extra input for general mixed intersections is 
the product--Schur Hodge--Riemann theorem and the exact fiber
product representation of kappa products.

\section{Schur conventions and the positivity theorems}
\label{sec:positivity}

All theorem, remark, equation, and page numbers attached to
\cite{RT} in this paper refer to arXiv:2106.11285v1, which is
listed together with the published chapter in the bibliography.
This convention fixes the numbering of the sources used below.

\subsection{Schur classes and a rank-two identity}

The following is the classical Chern-class form of the
dual Jacobi--Trudi definition; see \cite[Chapter I, Sections 2--3]{Macdonald}
and \cite[Section 3]{RT}.

\begin{definition}\label{def:schur}
For a partition $\mu=(\mu_1\ge\cdots\ge\mu_l\ge0)$, define
\begin{equation}\label{eq:schur-definition}
  s_\mu(E)=\det\bigl(c_{\mu_i+j-i}(E)\bigr)_{1\le i,j\le l}.
\end{equation}
Its codimension is $|\mu|=\sum_i\mu_i$. The empty partition gives $1$.
\end{definition}

This is the Chern-determinant convention used in both \cite{LX} and
\cite{RT}. It is conjugate to another common convention in symmetric
function theory. In particular,
\begin{equation}\label{eq:schur-examples}
  s_{(b)}(E)=c_b(E),\qquad
  s_{(1,1)}(E)=c_1(E)^2-c_2(E).
\end{equation}
For a rank-$e$ bundle the relevant rank bound in this convention is
$\mu_1\le e$; partitions with larger first part give zero.

The next lemma is a classical symmetric-function identity. 

\begin{lemma}[Classical rank-two Schur identity]\label{lem:rank-two}
Let $L_A,L_B$ be line bundles with first Chern classes $A,B$.
For every $r\ge0$,
\begin{equation}\label{eq:complete-schur}
  s_{(1^r)}(L_A\oplus L_B)
    =h_r(A,B):=\sum_{j=0}^r A^{r-j}B^j,
\end{equation}
where $h_0=1$. Moreover,
\begin{equation}\label{eq:key-polynomial}
  (A-B)^2h_r(A,B)
  =A^{r+2}+B^{r+2}-A^{r+1}B-AB^{r+1}.
\end{equation}
\end{lemma}

\begin{proof}
For $E=L_A\oplus L_B$ one has
$c_1(E)=A+B$, $c_2(E)=AB$, and $c_j(E)=0$ for $j\ge3$.
The determinant for $s_{(1^r)}(E)$ is tridiagonal, with diagonal
$c_1(E)$, upper diagonal $c_2(E)$, and lower diagonal $1$.
Writing its value as $H_r$, expansion along the first row gives
\[
  H_0=1,\quad H_1=A+B,\quad
  H_r=(A+B)H_{r-1}-ABH_{r-2}\quad(r\ge2).
\]
The polynomials $h_r=\sum_{j=0}^r A^{r-j}B^j$ satisfy exactly the same
initial conditions and recurrence, as is seen by multiplying out the
right-hand side. Induction proves \eqref{eq:complete-schur}.
There is no upper bound on $r$: the largest part of $(1^r)$ is $1$.

The telescoping identity
\[
  (A-B)h_r(A,B)=A^{r+1}-B^{r+1}
\]
holds in every commutative ring. Multiplying once more by $A-B$
proves \eqref{eq:key-polynomial}, including $r=0$.
\end{proof}

\subsection{Numerical positivity and weak Hodge--Riemann}

A vector bundle $E$ on a projective variety is called nef if
$\OO_{\mathbb P(E)}(1)$ is nef, with $\mathbb P(E)$ parametrizing
one-dimensional quotients; see \cite[Chapter 6]{Lazarsfeld}.
We use the standard facts that pullback, direct sum, and extension
preserve nefness, recalled with proof in Lemma~\ref{lem:nef-extension}.
On a proper moduli stack,
nefness is understood by pullback to complete test curves. These
properties consequently apply after the scheme covers used below.

We quote the nef product form of numerical Schur positivity from
\cite[Remark 5.4]{RT}.

\begin{theorem}[Numerical Schur positivity, Ross--Toma]\label{thm:schur-positive}
Let $Y$ be a smooth projective variety of dimension $N$, let
$E_1,\ldots,E_s$ be nef vector bundles, and let $\mu^1,\ldots,\mu^s$
be partitions with $\sum_j|\mu^j|=N$. Then
\begin{equation}\label{eq:schur-positive}
  \int_Y\prod_{j=1}^s s_{\mu^j}(E_j)\ge0.
\end{equation}
\end{theorem}

It includes products
of different bundles and repeated copies of the same bundle. A nef
divisor factor is $s_{(1)}(L)=c_1(L)$; its powers can be included as
repeated factors. Degree-zero factors may be omitted.

For a smooth projective $N$-fold with $N\ge2$ and
$\eta\in H^{N-2,N-2}(Y;\RR)$, write
\[
  Q_\eta(\alpha,\beta)=\int_Y\alpha\beta\eta,
  \qquad \alpha,\beta\in H^{1,1}(Y;\RR).
\]
The Hodge--Riemann property in the form needed here means that this
real symmetric form has signature $(+,-,\ldots,-)$.

The following theorem is quoted from
\cite[Theorem 7.4]{RT}.

\begin{theorem}[Ross--Toma]\label{thm:RT}
Let $Y$ be a smooth projective $N$-fold, $N\ge2$, and let $E_j$ be
nef vector bundles. If
\[
  \eta=\prod_j s_{\mu^j}(E_j),\qquad
  \sum_j|\mu^j|=N-2,
\]
then $\eta$ lies in the closure of the Hodge--Riemann cone.
In particular, $Q_\eta$ has at most one positive eigenvalue.
\end{theorem}

The theorem is about a
\emph{product of Schur classes}; one cannot replace this hypothesis
by an arbitrary sum of individually positive classes. The empty
product in dimension two is covered by the classical Hodge-index
theorem.

The following is the standard linear-algebra argument underlying
the Hodge-index inequality; it is the argument used in
\cite[Theorem 10.1]{RT}.
We reproduce it, including the degenerate case.

\begin{lemma}[The Hodge-index inequality]\label{lem:linear-algebra}
Let $Q$ be a real symmetric bilinear form with at most one positive
eigenvalue. If $Q(b,b)\ge0$, then
\begin{equation}\label{eq:reverse-CS}
  Q(a,a)Q(b,b)\le Q(a,b)^2
\end{equation}
for every $a$.
\end{lemma}

\begin{proof}
If $Q(b,b)=0$, the left side is zero, so the assertion is immediate.
Otherwise let
\[
  a_0=a-\frac{Q(a,b)}{Q(b,b)}b.
\]
Then $Q(a_0,b)=0$. If $Q(a_0,a_0)>0$, the restriction to the span
of $a_0,b$ would be positive definite, contradicting the bound on
the positive index of $Q$. Hence $Q(a_0,a_0)\le0$. Expanding this
inequality and multiplying by the positive number $Q(b,b)$ gives
\eqref{eq:reverse-CS}. The argument does not require nondegeneracy.
\end{proof}

Combining Theorems~\ref{thm:schur-positive} and \ref{thm:RT} with
Lemma~\ref{lem:linear-algebra} gives, for nef $A,B$,
\begin{equation}\label{eq:mixed-HI}
  \left(\int_Y A^2\eta\right)\left(\int_Y B^2\eta\right)
  \le\left(\int_Y AB\eta\right)^2.
\end{equation}
Indeed, numerical Schur positivity verifies the required nonnegative
$B$-square. This is the instance of \cite[Theorem 10.1]{RT} used below.

\section{The symmetric balancing lemma}
\label{sec:balancing}

The next lemma is a deduction proved here from
\cite[Theorems 7.4 and 10.1]{RT} and the classical identity in
Lemma~\ref{lem:rank-two}. For the pure-divisor,
strictly positive setting, the symmetric log-concavity approach
appears in \cite[Sections 2.1 and 3]{Schmitt}.

\begin{lemma}\label{lem:balancing}
Let $Y$ be a smooth projective $N$-fold. Let $A,B$ be first Chern
classes of nef line bundles, and let $\Omega$ be a product of Schur
classes of nef bundles. Suppose
\[
  \codim\Omega=N-s,\qquad
  \int_Y A^iB^{s-i}\Omega
    =\int_Y A^{s-i}B^i\Omega\quad(0\le i\le s).
\]
Then for nonnegative integers $p<q$ with $p+q=s$,
\begin{equation}\label{eq:balancing}
  \int_Y A^pB^q\Omega
    \le \int_Y A^{p+1}B^{q-1}\Omega.
\end{equation}
\end{lemma}

\begin{proof}
For $q=p+1$, the two expressions coincide by the assumed symmetry.
Otherwise set $r=q-p-2\ge0$ and define
\begin{equation}\label{eq:eta}
  \eta=(AB)^p h_r(A,B)\Omega.
\end{equation}
By Lemma~\ref{lem:rank-two}, $h_r(A,B)$ is a Schur class of the
nef bundle $L_A\oplus L_B$. Thus $\eta$ is a product of Schur
classes of nef bundles, not merely an unstructured sum of classes.
Its codimension is
\[
  2p+r+(N-s)=2p+(q-p-2)+(N-p-q)=N-2.
\]

The polynomial $(AB)^p h_r(A,B)$ is symmetric in $A,B$.
The numerical symmetry in the hypothesis therefore implies
\[
  u:=\int_Y A^2\eta=\int_Y B^2\eta.
\]
Put $v=\int_Y AB\eta$. Theorem~\ref{thm:schur-positive} gives
$u,v\ge0$, and \eqref{eq:mixed-HI} gives $u^2\le v^2$.
It follows that $u\le v$, even if one or both are zero. Hence
\begin{equation}\label{eq:negative-square}
  \int_Y(A-B)^2\eta=2u-2v\le0.
\end{equation}

Multiply \eqref{eq:key-polynomial} by $(AB)^p\Omega$.
Since $p+r+2=q$, integration gives
\begin{align}
  \int_Y(A-B)^2\eta
  & =\int_Y
     \bigl(A^qB^p+A^pB^q-A^{q-1}B^{p+1}
                   -A^{p+1}B^{q-1}\bigr)\Omega \notag\\
  & =2\left(\int_Y A^pB^q\Omega
          -\int_Y A^{p+1}B^{q-1}\Omega\right).
     \label{eq:difference-identity}
\end{align}
In the second line, the first pair of integrals and the second pair
are equal by symmetry. Equations \eqref{eq:negative-square} and
\eqref{eq:difference-identity} prove the result.
\end{proof}

\begin{remark}\label{rem:zeros}
The proof never divides by an intersection number. It therefore
handles arbitrary vanishing fixed weights, including products of
Hodge classes that vanish. This avoids a possible gap in a proof
that argues from nonnegative symmetric log-concavity alone:
an abstract nonnegative log-concave sequence can have internal zeros.
No regularization or strict-positivity assertion is needed here.
\end{remark}

\begin{remark}\label{rem:numerical-symmetry}
Only equality of numerical intersections is needed on $Y$.
An automorphism interchanging $A$ and $B$ on a downstairs space is
sufficient if all intersections upstairs are one common positive
multiple of those downstairs. The smooth projective model itself
need not carry a lifting of the automorphism.
\end{remark}

\section{Nef bundles on the compactified moduli space}
\label{sec:nef-moduli}

\subsection{Cotangent lines and the universal logarithmic line}
\label{subsec:log-line}

We use the standard nefness theorem for cotangent line bundles on
$\Mbar_{g,n}$; see \cite[Chapter XIV, Corollary 5.14]{ACG}.

The universal-curve identification and logarithmic kappa
convention are standard; see \cite{AC,Knudsen} and
\cite[Chapter XIV]{ACG}. The following lemma records the
associated line-bundle identity, together with its usual local
proof; nefness uses the cited cotangent-line theorem.

\begin{lemma}[Standard universal-curve identity]\label{lem:log-line}
Under $C\simeq\Mbar_{g,n+1}$, the line bundle
$\omega_\pi(\sum_i D_i)$ is the cotangent line bundle at the last
marking. In particular, $\ell=\psi_{n+1}$ is nef, and
\begin{equation}\label{eq:kappa-zero}
  \kappa_0=2g-2+n.
\end{equation}
\end{lemma}

\begin{proof}
A point of the universal curve is an original stable marked curve
together with a point of that curve; the point is allowed to lie
at a marking or a node. The identification with $\Mbar_{g,n+1}$
is obtained by inserting a rational component when necessary so
that the new marking is distinct and smooth.

Where the moving point is smooth and different from the existing
markings, its cotangent line is the fiber of $\omega_\pi$.
When it approaches the $i$th marking, the stabilization inserts
a rational component carrying the two markings. To calculate the
correction, choose a coordinate $z$ with the old marking at $z=0$
and write the moving marking as $z=u$. On the bubble use $v=z/u$:
the old marking is $v=0$, the new marking is $v=1$, and the attaching
node is $v=\infty$. The relative cotangent frame at the new marking
is $dv=dz/u$. Thus it is the old dualizing frame with one simple
pole along $u=0$, giving exactly the factor $\OO_C(D_i)$.
The collision divisors $D_i$ are disjoint. The locus of points
that are nodes of the fibers has
codimension two in the total universal curve. Since
$C\simeq\Mbar_{g,n+1}$ is smooth as a stack, an isomorphism of line
bundles away from a codimension-two locus extends uniquely. These
observations give
\[
  \mathcal L_{n+1}\simeq\omega_\pi(D_1+\cdots+D_n).
\]
Nefness follows from the cotangent-line theorem. Finally, the degree
of this logarithmic dualizing line on every complete fiber of
$\pi$ is $2g-2+n$. Its pushforward in codimension zero is therefore
the scalar in \eqref{eq:kappa-zero}.
\end{proof}

\subsection{Nefness of the full Hodge bundle, including the boundary}

We recall the semipositivity input in the form needed here.
Its original source over a curve is Fujita
\cite[Main Theorem (0.6)]{Fujita}; the extension to
higher-dimensional bases with unipotent local monodromy is
Kawamata \cite[Theorem 5]{Kawamata}.

\begin{theorem}[Fujita--Kawamata semipositivity]\label{thm:semipositivity}
For a generically smooth semistable family of connected curves over
a smooth projective curve, the direct image of the relative dualizing
sheaf is a nef vector bundle. The same conclusion holds for a generically smooth
stable family, after identifying the Hodge direct image with that
of a semistable resolution.
\end{theorem}

This is the curve case of Fujita--Kawamata semipositivity. A statement
with smooth total space, simple normal crossing degeneracy, and
unipotent local monodromies is \cite[Theorem 1.1]{CK}; it includes
the semistable curve situation. At a smoothing node of a stable
family over a nonsingular base curve, the local total space has
equation $xy=t^a$ for some $a\ge1$, up to units and etale localization.
Its singularities are rational double points. Resolving them, and
if necessary making a finite base change followed by semistable
resolution, does not change the corresponding Hodge direct image:
the pushforward of the resolved dualizing sheaf is the original
dualizing sheaf. The nodal monodromies are unipotent. These facts
justify the stable-family formulation above, and nefness descends
from a finite surjective base change.

The following are standard permanence properties of nef vector
bundles; see \cite[Chapter 6]{Lazarsfeld}. The proof below gives
the curve-test arguments needed here, so no new positivity
theorem is being asserted.

\begin{lemma}[Standard nefness properties]\label{lem:nef-extension}
Extensions and direct sums of nef vector bundles are nef. If a
vector bundle becomes nef after a finite surjective map of smooth
projective curves, then it was nef before that map.
\end{lemma}

\begin{proof}
For the extension assertion, test a sequence
$0\to F\to E\to G\to0$ on a smooth projective curve after every
curve pullback, and let $E\twoheadrightarrow Q$ be a quotient line
bundle. If $F\to Q$ is zero, then $Q$ is a quotient of $G$, so
$\deg Q\ge0$. Otherwise its image is $Q(-T)$ for an effective
divisor $T$; this image is a quotient line bundle of $F$.
Thus $\deg Q(-T)\ge0$, and consequently $\deg Q\ge0$.
The quotient-line criterion for nefness proves the result.
A direct sum is a split extension. For descent, given any test
curve downstairs, take a component of its fiber product with the
finite cover that dominates it, and normalize that component.
The pullback of a proposed quotient line has degree equal to its
original degree times the positive cover degree. Nonnegativity
upstairs therefore implies nonnegativity downstairs.
\end{proof}

The next proposition is a standard consequence of
Fujita--Kawamata semipositivity
\cite[Main Theorem (0.6)]{Fujita},
\cite[Theorem 5]{Kawamata}, rather than a new nefness result.
We include the boundary reduction to explain why it applies
to the full compactification; the clutching constructions are
those of \cite{ACG,Knudsen}.

\begin{proposition}[Nefness of the Hodge bundle]\label{prop:hodge-nef}
The Hodge vector bundle $\EE$ is nef on the entire stack
$\Mbar_{g,n}$.
\end{proposition}

\begin{proof}
It suffices to test pullbacks to smooth projective curves, permitting
finite covers of those curves. A test curve whose generic point
parametrizes a smooth curve is handled by
Theorem~\ref{thm:semipositivity}. We treat test curves whose generic
point is in the boundary.

Let $B\to\Mbar_{g,n}$ be such a test curve and let $\Gamma$ be the
dual graph of its generic stable curve. The proper clutching map
\[
  \xi_\Gamma:
  \prod_{v\in V(\Gamma)}\Mbar_{g_v,n_v}\longrightarrow\Mbar_{g,n}
\]
has image the closure of the corresponding boundary stratum.
After a finite cover of $B$, we may lift the map to this product
and label its vertices and branches. Each vertex then determines
a family of connected stable curves $f_v:C_v\to B$ whose generic
fiber is smooth. Special fibers of $C_v$ may themselves acquire
nodes; this is allowed.

Let $f:C\to B$ be the family obtained by the prescribed gluing.
Normalize only the nodes corresponding to the edges of the fixed
graph $\Gamma$, so the partial normalization is the disjoint union
of the $C_v$. The structure-sheaf normalization sequence is
\begin{equation}\label{eq:normalization-O}
  0\longrightarrow\OO_C
   \longrightarrow\nu_*\OO_{\coprod_v C_v}
   \longrightarrow\bigoplus_{e\in E(\Gamma)}\OO_{B_e}
   \longrightarrow0,
\end{equation}
where $B_e\simeq B$ is the node section for edge $e$ and the last
map is the difference of the two branch values. Each component
family has connected fibers, so its degree-zero direct image is
$\OO_B$. The beginning of the pushed-forward sequence contains
the incidence complex
\[
  0\longrightarrow\OO_B\longrightarrow
  \OO_B^{V(\Gamma)}\longrightarrow\OO_B^{E(\Gamma)}.
\]
Its cokernel is the constant bundle
$H^1(\Gamma,\CC)\otimes\OO_B$. Since the node sections have no
higher direct images, \eqref{eq:normalization-O} gives an exact
sequence of vector bundles
\[
  0\longrightarrow H^1(\Gamma,\CC)\otimes\OO_B
   \longrightarrow R^1f_*\OO_C
   \longrightarrow\bigoplus_v R^1(f_v)_*\OO_{C_v}
   \longrightarrow0.
\]
Relative duality, followed by dualizing this sequence, yields
\begin{equation}\label{eq:hodge-boundary}
  0\longrightarrow\bigoplus_v(f_v)_*\omega_{f_v}
   \longrightarrow f_*\omega_f
   \longrightarrow H_1(\Gamma,\CC)\otimes\OO_B
   \longrightarrow0.
\end{equation}
This remains exact over special fibers where the component families
degenerate further: the construction normalizes only the fixed
gluing nodes, and each $C_v$ remains a connected family.

Every bundle $(f_v)_*\omega_{f_v}$ is nef by
Theorem~\ref{thm:semipositivity}, and the rightmost bundle of
\eqref{eq:hodge-boundary} is trivial. Lemma~\ref{lem:nef-extension}
now proves that $f_*\omega_f$ is nef and that this conclusion
descends under the finite cover of $B$. Thus every complete-curve
test is nonnegative, proving the proposition.
\end{proof}

\begin{remark}
Nefness of $\lambda_1=\det\EE$ alone would not suffice for this
proof. Proposition~\ref{prop:hodge-nef} concerns the vector bundle
itself, which is required for applying Schur positivity to arbitrary
$c_b(\EE)$ and their products. In genus zero the Hodge bundle has
rank zero; positive-degree Hodge factors then vanish and are handled
separately rather than by a positive-rank assertion.
\end{remark}

\section{Fiber powers and the exact kappa lifting formula}
\label{sec:fiber}

For an integer $m\ge0$, put
\[
  Y_m=\underbrace{C\times_X\cdots\times_X C}_{m\text{ factors}},
  \qquad f_m:Y_m\to X,
  \qquad \ell_j=\pr_j^*\ell.
\]
We set $Y_0=X$, with $f_0=\id$. The morphism $\pi$ is flat and
projective of relative dimension one, so $f_m$ is flat and
projective of relative dimension $m$. In particular,
\begin{equation}\label{eq:fiber-dimension}
  \dim Y_m=D+m.
\end{equation}

The next lemma is an application of standard flatness facts,
not an additional moduli-space theorem. The going-down
property used in its proof is \cite[Tag 00HS]{Stacks}.

\begin{lemma}[Flatness and generic integrality]\label{lem:fiber-integral}
The fiber power $Y_m$ has an irreducible, generically reduced
fundamental cycle. Its generic fiber over $X$ is geometrically
integral.
\end{lemma}

\begin{proof}
The generic curve over $X$ is smooth and geometrically connected,
hence geometrically integral. A product of geometrically integral
curves is geometrically integral, giving the generic-fiber claim.
There are no additional vertical irreducible components: flatness
over the integral base implies, locally on etale charts, that
minimal primes of the total space contract to minimal primes of
the base. Thus every irreducible component dominates the base.
There is only one such component by generic integrality. The total
space is generically reduced since it is smooth on the dense locus
over smooth curves. These facts suffice for the assertion about
its fundamental cycle and for all subsequent pushforwards.
\end{proof}

The following identity is a formal consequence of the definition
of kappa classes, proper--flat base change, and the projection
formula. For those intersection-theoretic operations see
\cite[Sections 1.7 and 3.2]{Fulton}; their extension to the
Deligne--Mumford stack setting is given in \cite{Vistoli}.
We supply the induction for this particular fiber product.

\begin{proposition}[Exact product lifting]\label{prop:kappa-lift}
For all $a_1,\ldots,a_m\ge0$,
\begin{equation}\label{eq:push-lift}
  (f_m)_*\left(\prod_{j=1}^m\ell_j^{a_j+1}\cap[Y_m]\right)
  =\left(\prod_{j=1}^m\kappa_{a_j}\right)\cap[X].
\end{equation}
Consequently, for every additional class $\alpha$ of complementary
degree,
\begin{equation}\label{eq:int-lift}
  \int_X\alpha\prod_{j=1}^m\kappa_{a_j}
  =\int_{Y_m}f_m^*\alpha\prod_{j=1}^m\ell_j^{a_j+1}.
\end{equation}
\end{proposition}

\begin{proof}
For $m=0$ the formula is the identity; for $m=1$ it is the
definition of $\kappa_{a_1}$. For the inductive step consider
the Cartesian diagram
\[
\begin{CD}
Y_m=Y_{m-1}\times_X C @>{r}>> C\\
@V{q}VV @VV{\pi}V\\
Y_{m-1} @>>{f_{m-1}}> X.
\end{CD}
\]
Proper pushforward commutes with flat base change
\cite[Proposition 1.7]{Fulton}, and
$f_{m-1}$ is flat. Consequently the pushforward by $q$ of
$r^*(\ell^{a_m+1})$ against the fundamental cycle is
$f_{m-1}^*\kappa_{a_m}$ against $[Y_{m-1}]$. Applying the
projection formula first to $q$ and then to $f_{m-1}$ gives
\begin{align*}
  (f_m)_*\left(\prod_{j=1}^m\ell_j^{a_j+1}\cap[Y_m]\right)
  &=(f_{m-1})_*\left(
     \prod_{j=1}^{m-1}\ell_j^{a_j+1}
     f_{m-1}^*\kappa_{a_m}\cap[Y_{m-1}]\right)\\
  &=\kappa_{a_m}\cap
     (f_{m-1})_*\left(
        \prod_{j=1}^{m-1}\ell_j^{a_j+1}\cap[Y_{m-1}]\right).
\end{align*}
The induction hypothesis proves \eqref{eq:push-lift}.
The formula \eqref{eq:int-lift} follows by one more application
of the projection formula and taking degrees.
\end{proof}

\begin{remark}\label{rem:not-marked-space}
Auxiliary points in $Y_m$ may coincide, and may lie at original
markings or nodes. Thus $Y_m$ is not the moduli space
$\Mbar_{g,n+m}$. All the classes $\ell_j$ are pulled back from
the same original universal curve. This is why
\eqref{eq:push-lift} is an exact product formula with no diagonal
or successive-forgetting corrections.
\end{remark}

\section{Reduction to smooth projective varieties}
\label{sec:stack}

The positivity theorems in Section~\ref{sec:positivity} are
statements for smooth projective varieties. We now justify their
application to the possibly singular fiber-power stacks.
The next proposition combines the standard finite scheme cover
\cite[Theorem 7.4.2]{ACV}, Hironaka's resolution theorem
\cite{Hironaka}, and the projection formula for stack
intersection theory \cite{Vistoli}. It is a reduction from
these results, not a new covering or resolution theorem.

\begin{proposition}\label{prop:smooth-cover}
For each $m\ge0$ there is a smooth projective irreducible variety
$V_m$ and a proper generically finite map $\rho:V_m\to Y_m$ such
that
\begin{equation}\label{eq:cover-degree}
  \rho_*[V_m]=e[Y_m]\quad\text{for some }e\in\QQ_{>0}.
\end{equation}
For every top-degree product $\gamma$ of Chern classes of vector
bundles on $Y_m$,
\begin{equation}\label{eq:cover-integrals}
  \int_{V_m}\rho^*\gamma=e\int_{Y_m}\gamma.
\end{equation}
All bundles used in the proof remain nef after pullback to $V_m$.
\end{proposition}

\begin{proof}
Choose a finite flat surjective cover $Z\to X$ by an irreducible
complete scheme, as in \cite[Theorem 7.4.2]{ACV}; this is the
cover used explicitly in \cite[Proposition 3.2]{Schmitt}. The
coarse moduli space $X_{\mathrm{coarse}}$ is projective. The
composite $Z\to X_{\mathrm{coarse}}$ is proper and quasi-finite,
hence finite by \cite[Tag 02LS]{Stacks}, so $Z$ is projective.

Because the universal curve is representable and projective,
the base change
\[
  W_m=Y_m\times_X Z
\]
is a projective scheme over $Z$, and hence projective over $\CC$.
Its generic fiber over $Z$ is geometrically integral. Take the
reduced main component that dominates $Z$ and resolve it
projectively in characteristic zero, using \cite{Hironaka}.
The result is a smooth
projective irreducible variety $V_m$. Its map $\rho$ to $Y_m$ is
proper, dominant, and generically finite.

Pushforward of a fundamental cycle under such a map gives
\eqref{eq:cover-degree}. Here $e$ means the degree in rational
Chow theory of stacks \cite{Vistoli}; any generic stabilizer factor is included
in $e$, so no identification with a naive function-field degree
is required. The only essential fact is $e>0$.
Projection formula gives
\[
  \rho_*\bigl(\rho^*\gamma\cap[V_m]\bigr)
     =\gamma\cap\rho_*[V_m]
     =e\,\gamma\cap[Y_m],
\]
and taking degrees proves \eqref{eq:cover-integrals}.
Finally pullback preserves nefness of line and vector bundles.
\end{proof}

The following corollary is the descent deduction proved here
from Proposition~\ref{prop:smooth-cover} and
\cite[Remark 5.4, Theorems 7.4 and 10.1]{RT}.

\begin{corollary}\label{cor:downstairs}
Numerical Schur positivity applies on $Y_m$ to products of
Schur classes of nef bundles. The Hodge-index inequality and
the balancing lemma also apply to such intersection numbers
on $Y_m$, whenever their stated symmetry hypotheses hold there.
\end{corollary}

\begin{proof}
Pull all bundles back along $\rho$ and apply the results on $V_m$.
By \eqref{eq:cover-integrals}, each linear inequality is multiplied
by $e$ and each quadratic Hodge-index inequality by $e^2$.
These positive factors cancel. If $\sigma:Y_m\to Y_m$ exchanges
two classes and fixes the remaining ones, then
\[
  \int_{Y_m}\sigma^*\gamma=\int_{Y_m}\gamma
\]
by invariance of the fundamental cycle. Multiplication by $e$
transfers this equality to the integrals on $V_m$, providing
exactly the numerical symmetry in Lemma~\ref{lem:balancing}.
There is no need to lift $\sigma$ to $V_m$.
\end{proof}

\section{Proof of the three inequalities}
\label{sec:proof}

If any fixed Hodge index $b_t$ exceeds $g$, the fixed product
vanishes and all assertions involving it are immediate. We may
therefore omit such trivial cases, as well as factors
$\lambda_0=1$. Every remaining Hodge factor is a Schur class of
the nef bundle $\EE$.

\subsection{Psi balancing}

Let $P_\psi$ be as in Theorem~\ref{thm:main}(i). Lift its $m$
kappa factors to $Y_m$. On this stack put
\begin{align*}
  A&=f_m^*\psi_1,\qquad B=f_m^*\psi_2,\\
  \Omega&=
     \prod_{i=3}^n(f_m^*\psi_i)^{d_i}
     \prod_{j=1}^m\ell_j^{a_j+1}
     \prod_{t=1}^k c_{b_t}(f_m^*\EE).
\end{align*}
These are first Chern classes and Schur classes of nef bundles
by Section~\ref{sec:nef-moduli}. The top-degree constraint is
\[
  p+q+\sum_{i=3}^n d_i+\sum_j a_j+\sum_t b_t=D.
\]
Thus
\[
  \codim\Omega=D-p-q+m=\dim Y_m-p-q,
\]
as required by the balancing lemma.

The transposition of original markings $1,2$ acts on $X$, on
its universal curve, and on every fiber power. It exchanges
$A$ and $B$. It fixes the cotangent classes at all other
markings, the Hodge bundle, and the logarithmic line
$\omega_\pi(\sum_iD_i)$, since the sum of the sections is
unchanged. Hence it fixes $\Omega$ and supplies all the
numerical symmetries required by Lemma~\ref{lem:balancing}.
Corollary~\ref{cor:downstairs} therefore gives
\[
  \int_{Y_m} A^pB^q\Omega
  \le\int_{Y_m} A^{p+1}B^{q-1}\Omega.
\]
Using \eqref{eq:int-lift} on both sides is exactly
\eqref{eq:main-psi}. This proves part (i).

\subsection{Kappa balancing}

Let $P$ be as in Theorem~\ref{thm:main}(ii). Use $Y_{m+2}$,
reserving its first two factors for the two distinguished
kappa insertions. Define
\begin{align*}
  A&=\ell_1,\qquad B=\ell_2,\\
  \Omega&=\prod_{i=1}^n(f_{m+2}^*\psi_i)^{d_i}
       \prod_{j=1}^m\ell_{j+2}^{a_j+1}
       \prod_{t=1}^k c_{b_t}(f_{m+2}^*\EE).
\end{align*}
Again every factor is a Schur class of a nef bundle. Now the
top-degree condition reads
\[
  p+q+\sum_i d_i+\sum_j a_j+\sum_t b_t=D,
\]
and consequently
\[
  \codim\Omega=D-p-q+m
    =\dim Y_{m+2}-(p+1)-(q+1).
\]
Interchanging the first two universal-curve factors exchanges
$A,B$ and preserves $\Omega$; all base pullbacks and the other
fiber factors are unchanged.

Since $p+1<q+1$, apply Lemma~\ref{lem:balancing} to the
exponent pair $(p+1,q+1)$, using
Corollary~\ref{cor:downstairs}. It yields
\[
  \int_{Y_{m+2}} A^{p+1}B^{q+1}\Omega
  \le\int_{Y_{m+2}} A^{p+2}B^q\Omega.
\]
Formula \eqref{eq:int-lift} identifies the first side with
$\int_X\kappa_p\kappa_qP$ and the second with
$\int_X\kappa_{p+1}\kappa_{q-1}P$. This proves part (ii).

The case $p=0$ is included. Its distinguished fiber exponent
is $1$, in accordance with $\kappa_0=\pi_*\ell$. The case
$q=p+1$ is equality, either directly or by the equality case
of the lemma.

\subsection{Lambda balancing}

If $q=p+1$, equality is automatic. Suppose $q\ge p+2$.
Then $\mu=(q-1,p+1)$ is a partition, and the defining
two-by-two determinant gives
\begin{align}
  s_{(q-1,p+1)}(\EE)
  &=\det\begin{pmatrix}
       c_{q-1}(\EE)&c_q(\EE)\\
       c_p(\EE)&c_{p+1}(\EE)
      \end{pmatrix}\notag\\
  &=\lambda_{p+1}\lambda_{q-1}-\lambda_p\lambda_q.
     \label{eq:lambda-difference}
\end{align}
Its codimension is $p+q$.

Lift the $m$ kappa factors in $P$ to $Y_m$. Equations
\eqref{eq:int-lift} and \eqref{eq:lambda-difference} give
\begin{align*}
 &\int_X
   (\lambda_{p+1}\lambda_{q-1}-\lambda_p\lambda_q)P\\
 &\quad=
   \int_{Y_m}s_{(q-1,p+1)}(f_m^*\EE)
      \prod_{t=1}^k c_{b_t}(f_m^*\EE)
      \prod_{i=1}^n(f_m^*\psi_i)^{d_i}
      \prod_{j=1}^m\ell_j^{a_j+1}.
\end{align*}
The integrand is a product of Schur classes of nef bundles of
total codimension $D+m=\dim Y_m$. It has nonnegative integral
by Theorem~\ref{thm:schur-positive} and
Corollary~\ref{cor:downstairs}. This is part (iii).

For clarity about the rank bounds: if $q>g+1$, then both
$\lambda_q$ and $\lambda_{q-1}$ vanish, so both sides are
zero. If $q=g+1$ and $q\ge p+2$, then $q-1=g$ and
$p+1\le g$; the same Schur determinant remains valid and
proves nonnegativity of the surviving product. If $q\le g$,
the largest partition part is at most $g$, as required.
In genus zero every positive-degree Hodge insertion is zero.

This completes the proof of Theorem~\ref{thm:main}.\hfill$\square$

\section{Further consequences and consistency checks}
\label{sec:consequences}

\subsection{Pairwise log-concavity}

The same positivity theorem also gives a useful stronger local
inequality, although it is not needed for the direct balancing
proof. The following is an immediate application of
\cite[Theorems 7.4 and 10.1]{RT}; we give the substitution
explicitly, rather than claim a separate new log-concavity
theorem. The pure-psi version is the log-concavity property
in \cite[Proposition 3.2]{Schmitt}.

\begin{proposition}\label{prop:log-concavity}
In the geometric setting of Lemma~\ref{lem:balancing}, even
without assuming symmetry, the numbers
\[
  I_i=\int_Y A^iB^{s-i}\Omega\qquad(0\le i\le s)
\]
are nonnegative and satisfy
\[
  I_i^2\ge I_{i-1}I_{i+1}\qquad(1\le i\le s-1).
\]
\end{proposition}

\begin{proof}
Nonnegativity is Theorem~\ref{thm:schur-positive}.
For an interior index put
$\eta_i=A^{i-1}B^{s-i-1}\Omega$. This is a product of
Schur classes of nef bundles of codimension $N-2$.
In \eqref{eq:mixed-HI}, the three relevant terms are
\[
  \int_Y AB\eta_i=I_i,\qquad
  \int_Y A^2\eta_i=I_{i+1},\qquad
  \int_Y B^2\eta_i=I_{i-1}.
\]
Substitution proves the assertion.
\end{proof}

Via fiber-power lifting, this applies to psi slices with all
other psi, kappa, and lambda insertions fixed. It also applies
to interior kappa slices after shifting each of the two
distinguished exponents by one. The proof of the main theorem
does not infer unimodality from this proposition; the direct
Schur argument already handles possible zero entries.

\subsection{Balanced extrema at fixed total degree}

The following consequence is deduced here from
Theorem~\ref{thm:main}. Its pure-psi specialization is
\cite[Theorem 1.2]{Schmitt}; the balancing-move argument
is the standard optimization argument also used in
\cite[Theorem 3.1 and Sections 3.1--3.4]{Schmitt}.

\begin{corollary}\label{cor:extrema}
Fix a monomial $P$ in psi, kappa, and lambda classes. In the
psi case, choose $r\ge1$ marked points not used by the psi
factors of $P$, and vary the exponents at those points. In
the kappa case, vary $r\ge1$ additional kappa factors while
keeping $P$ fixed. Fix the sum $S$ of the varying indices
so that the products are of top degree. A balanced list of
indices (all differing by at most one) gives a maximum among
those lists. A concentrated list $(S,0,\ldots,0)$ gives a
minimum. The assertions are weak: a maximum or minimum need
not be unique.
\end{corollary}

\begin{proof}
If two entries differ by at least two, replace $(p,q)$,
$p<q$, by $(p+1,q-1)$. The main theorem says that the value
does not decrease. The sum of squares of the entries decreases
by $2(q-p-1)>0$, so this process terminates at a balanced
list. All balanced lists are permutations of one another,
and symmetry gives the same value on them.

For the minimum, choose a largest entry $b$ and another
positive entry $a\le b$. Apply the balancing inequality to
$(a-1,b+1)$: its balanced successor is $(a,b)$, so replacing
$(a,b)$ by $(a-1,b+1)$ does not increase the value. Repeating
concentrates all mass in one entry. This gives the claimed
minimum.
\end{proof}

Pairwise balancing is a partial-order statement, often
expressed through majorization. It is not monotonicity under
an arbitrary total ordering of partitions. In particular,
pure-psi integral is not monotonic under length-first lexicographic order (see \cite{GNYZ}).

\section*{Disclosure of AI assistance}
\addcontentsline{toc}{section}{Source and preparation notes}

This manuscript was prepared with the assistance of GPT-6 Astra, 
which helped verify the relevant literature and organized
the proofs of the key algebro-geometric theorems in Sections \ref{sec:nef-moduli}-\ref{sec:stack}.

The numerical Schur positivity, product--Schur Hodge--Riemann,
and semipositivity theorems are quoted results. 
The rank-two Schur identity, Hodge-index linear algebra,
universal-curve identity, nefness permanence properties, and
flatness lemma are standard facts whose proofs are included
for completeness. The nefness of the compactified Hodge
bundle is a consequence of semipositivity and the
boundary normalization sequence. The kappa lifting formula
and smooth-model reduction are the applications of 
intersection and covering results explained in their proofs.


\begin{thebibliography}{99}
\addcontentsline{toc}{section}{References}

\bibitem{AC} E. Arbarello and M. Cornalba,
{\em Combinatorial and algebro-geometric cohomology classes on the
moduli spaces of curves}, J. Algebraic Geom. \textbf{5} (1996), 705--749.

\bibitem{ACG}
E. Arbarello, M. Cornalba and P. Griffiths,
\emph{Geometry of Algebraic Curves, Volume II},
Grundlehren der mathematischen Wissenschaften, vol.~268,
Springer, Heidelberg, 2011.

\bibitem{ACV}
D. Abramovich, A. Corti and A. Vistoli,
\emph{Twisted bundles and admissible covers},
Comm. Algebra \textbf{31} (2003), 3547--3618.


\bibitem{CK}
F. Catanese and Y. Kawamata,
\emph{Fujita decomposition over higher dimensional base},
Eur. J. Math. \textbf{5} (2019), 720-–728.

\bibitem{Faber} C. Faber, \emph{Algorithms for computing intersection numbers on moduli spaces of curves, with an application to the
class of the locus of Jacobians}, in New Trends in Algebraic
Geometry,
93--109, Cambridge University Press, 1999.

\bibitem{Fujita}
T. Fujita,
\emph{On K\"ahler fiber spaces over curves},
J. Math. Soc. Japan \textbf{30} (1978), 779--794.


\bibitem{Fulton}
W. Fulton,
\emph{Intersection Theory}, second edition,
Ergebnisse der Mathematik und ihrer Grenzgebiete (3), vol.~2,
Springer, Berlin, 1998.

\bibitem{GNYZ}
J. Guo, P. Norbury, D. Yang and D. Zagier, \emph{Combinatorics and large genus asymptotics of the Br\'ezin--Gross--Witten numbers}, arXiv:2412.20388.

\bibitem{Hironaka}
H. Hironaka,
\emph{Resolution of singularities of an algebraic variety
over a field of characteristic zero, I, II},
Ann. of Math. (2) \textbf{79} (1964), 109--203 and 205--326.

\bibitem{Kawamata}
Y. Kawamata,
\emph{Characterization of abelian varieties},
Compositio Math. \textbf{43} (1981), 253--276.

\bibitem{Knudsen}
F. F. Knudsen,
\emph{The projectivity of the moduli space of stable curves, II:
The stacks $M_{g,n}$},
Math. Scand. \textbf{52} (1983), 161--199.

\bibitem{Kontsevich} M. Kontsevich, \emph{Intersection theory on the moduli space of curves and the matrix
Airy function}, Comm. Math. Phys. \textbf{147} (1992), 1--23.

\bibitem{Lazarsfeld}
R. Lazarsfeld,
\emph{Positivity in Algebraic Geometry II},
Ergebnisse der Mathematik und ihrer Grenzgebiete (3), vol.~49,
Springer, Berlin, 2004. 

\bibitem{LX}
K. Liu and H. Xu,
\emph{New properties of the intersection numbers on moduli spaces
of curves},
Math. Res. Lett. \textbf{14} (2007), 1041--1054.

\bibitem{Macdonald}
I. G. Macdonald,
\emph{Symmetric Functions and Hall Polynomials}, second edition,
Oxford Mathematical Monographs,
Oxford University Press, Oxford, 1995.


\bibitem{RT}
J. Ross and M. Toma,
\emph{On Hodge-Riemann Cohomology Classes},
in \emph{Birational Geometry, K\"ahler--Einstein Metrics and Degenerations},
Springer Proceedings in Mathematics \& Statistics, vol. 409,
Springer, 2023, pp. 763--793.

\bibitem{Schmitt}
J. Schmitt,
\emph{Extremal descendant integrals on moduli spaces of curves:
an inequality discovered and proved in collaboration with AI},
arXiv:2512.14575v1 (2025).

\bibitem{Stacks}
The Stacks Project Authors,
\emph{The Stacks Project}.
Tags~00HS (going down for flat ring maps) and
02LS (proper and locally quasi-finite morphisms are finite).
\url{https://stacks.math.columbia.edu/tag/00HS};
\url{https://stacks.math.columbia.edu/tag/02LS}.

\bibitem{Vistoli}
A. Vistoli,
\emph{Intersection theory on algebraic stacks and on their moduli spaces},
Invent. Math. \textbf{97} (1989), 613--670.

\bibitem{Witten} E. Witten,
\emph{Two-dimensional gravity and intersection theory on moduli
space}, Surveys in Differential Geometry, vol.1, (1991) 243--310.

\end{thebibliography}
\end{document}